\documentclass[12pt]{amsart}

\usepackage{geometry}                
\usepackage[dvipdfmx]{graphicx}  
\usepackage{subfig}
\usepackage{labelfig}
\usepackage{amssymb}
\usepackage{color}
\newtheorem{theorem}{Theorem}[section]    
\newtheorem{lemma}[theorem]{Lemma}          
\newtheorem{proposition}[theorem]{Proposition}

\newtheorem{corollary}[theorem]{Corollary} 

\theoremstyle{definition}
\newtheorem{definition}[theorem]{Definition}

\newtheorem{remark}[theorem]{Remark}

\newtheorem{example}[theorem]{Example}   
\numberwithin{equation}{section}

\newcommand{\F}{\mathcal F_{ob} }
\newcommand{\cF}{\mathcal F_\xi }

\newcommand{\Z}{\mathbb{Z}}

\newcommand{\Aut}{{\rm Aut}}

\begin{document}

\title{Coverings of open books}
\author{Tetsuya Ito}
\address{Research Institute for Mathematical Sciences, Kyoto university, Kyoto, 606-8502, Japan}
\email{tetitoh@kurims.kyoto-u.ac.jp}
\urladdr{http://www.kurims.kyoto-u.ac.jp/~tetitoh/}
\author{Keiko Kawamuro}
\address{Department of Mathematics,   
The University of Iowa, Iowa City, IA 52242, USA}
\email{kawamuro@iowa.uiowa.edu}

\date{\today} 



\begin{abstract}
We study a coverings of open books  and virtually overtwisted contact manifolds using open book foliations. 
We show that open book coverings produces interesting examples such as transverse knots with depth grater than $1$. 
We also demonstrate explicit examples of virtually overtwisted open books.
\end{abstract}

\maketitle

\section{Introduction}

In the classification of contact structures on oriented $3$-manifolds there is a dichotomy between tight and overtwisted contact structures. 
The classification of overtwisted contact structures is reduced to homotopy theory by Eliashberg \cite{E}. 
This is not the case for tight contact structures and study of tight contact structures is an active topic in contact geometry.  
A tight contact structure is called {\em universally tight} if its universal cover is tight, and {\em virtually overtwisted} if it has a {\em finite} cover that is overtwisted. 
As a consequence of the geometrization, the fundamental groups of 3-manifolds are residually finite, which implies that every tight contact structure is either universally tight or virtually overtwisted (cf. \cite{H}). 
Namely, {\em universally} overtwisted is equilvalent to virtually overtwisted.

The idea of coverings plays important roles in many areas of mathematics, including study of contact structures. 
In this note we identify a covering map of contact manifolds with an open book covering map (see Section~\ref{sec2}), and study virtually overtwisted contact manifolds using open book foliations. 
Here is one of the results.

\noindent
{\bf Corollary~\ref{corA}.}
{\em Let $B$ be the binding of an open book $(S,\phi)$. Then the {\em depth} \cite{bo} of the binding is $1$ if and only if $\phi$ is not right-veering.}

In Section~\ref{sec4} we study examples of open books which have interesting properties. 
We give a family of planar open books that supports overtwisted, virtually overtwisted and universally tight contact structures. 
Some non-planar examples are also discussed. 

\begin{proposition}\label{key-example}
Let $S=S_{0, p+q}$ be a sphere with $p+q$ holes, where $p, q\geq 2$. 
Let $\alpha, \beta, \gamma \subset S$ be circles as shown in  Figure~\ref{sphere}. 
Let $\phi \in \Aut(S, \partial S)$ be a diffeomorphism given by 
$$\phi=T \circ {T_\alpha}^{n} \circ T_\beta \circ T_\gamma$$ 
where $T$ is the product of one positive Dehn twist along each of the $p+q$ boundary components and 
$T_\alpha$ is the positive Dehn twist along the curve $\alpha$. 
\begin{enumerate}
\item
If $n \leq -2$ then $(S, \phi)$ supports an overtwisted contact structure.
\item
If $n =-1$ then $(S, \phi)$ supports a virtually overtwisted tight contact structure.
\item
If $n \geq 0$ then $(S, \phi)$ supports a universally tight contact structure. 
\end{enumerate}
\begin{figure}[htbp]
\begin{center}
\SetLabels
(.2*.08) \scriptsize $p-1$\\
(.22*.27) \scriptsize $2$\\
(.22*.35) \scriptsize $1$\\
(.8*.08) \scriptsize $q-1$\\
(.74*.27) \scriptsize $2$\\
(.74*.35) \scriptsize $1$\\
(.5*.8) \scriptsize $\alpha$\\
(.45*.25) \scriptsize $\beta$\\
(.55*.25) \scriptsize $\gamma$\\
\endSetLabels
\strut\AffixLabels{\includegraphics*[height=50mm]{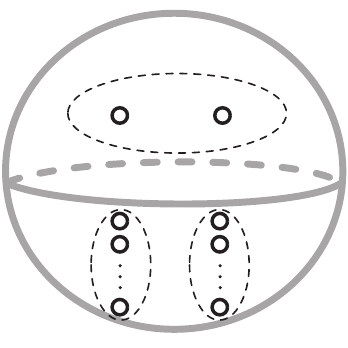}}
\caption{The planar surface $S$ with $p+q$ boundary components.}\label{sphere}
\end{center}
\end{figure}
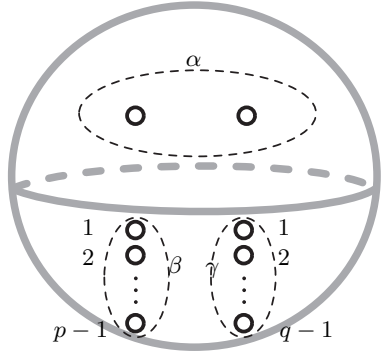
\end{proposition}

\section{Giroux correspondence and coverings}\label{sec2}
Let $S=S_{g,r}$ be an oriented genus $g$ surface with $r$ boundary components and $\phi \in \Aut(S, \partial S)$ be an orientation preserving diffeomorphism of $S$ fixing the boundary $\partial S$ pointwise. 
The pair $(S, \phi)$ is called an {\em abstract open book} (in this note the adjective ``abstract'' is omitted for simplicity) and $M_{(S, \phi)}$ denotes the closed oriented $3$-manifold obtained by gluing the mapping torus of $\phi$ and solid tori. 
See Etnyre's lecture note \cite{Et} for basics (and more) of open books. 
The Giroux correspondence \cite{G} states that there is a one-to-one correspondence between 
open books (up to positive stabilization) and contact manifolds (up to isotopy). 
We denote by $\xi_{(S, \phi)}$ the (isotopy class of) contact structure on the manifold $M_{(S, \phi)}$ {\em compatible with} (or we often say {\em supported by}) the open book $(S, \phi)$ via the Giroux correspondence.

Throughout this note a covering means a \emph{finite} covering.
Suppose that $\pi:\tilde S \to S$ is a covering map.

\begin{definition}
If there exists a diffeomorphism  $\tilde\phi \in \Aut(\tilde S, \partial\tilde S)$ satisfying $$\pi \circ \tilde\phi = \phi \circ \pi$$ then we call $(\tilde S, \tilde\phi)$ a {\em covering} of the open book $(S, \phi)$.
We write $\pi:(\tilde S, \tilde\phi) \to (S, \phi)$ abusing the notation and call it an {\em open book covering map}. 
$$
\begin{array}{rcl}
\tilde S & \stackrel{\tilde\phi}{\longrightarrow} & \tilde S\\
\pi\downarrow &  & \downarrow\pi\\
S & \stackrel{\phi}{\longrightarrow} & S
   \end{array}
$$   
\end{definition}

\begin{theorem}\label{prop1}
Let $\pi:(\tilde S, \tilde\phi) \to (S, \phi)$ be an open book covering map. 
Then the compatible contact structures for the open books, via the Giroux correspondence \cite{G}, yield a covering map
$$P:(M_{(\tilde S, \tilde\phi)}, \xi_{(\tilde S, \tilde\phi)}) \to (M_{(S, \phi)}, \xi_{(S, \phi)})$$
{\em compatible} with $\pi$, namely the restriction of $P$ to each page $\tilde S_t$  $(t\in[0,1])$ satisfies $P |_{\tilde S_t} = \pi$.
\end{theorem}

\begin{proof}
For simplicity we denote the covering space $(M_{(\tilde S, \tilde\phi)}, \xi_{(\tilde S, \tilde\phi)})$ by $(\tilde M, \tilde\xi)$, and the base space $(M_{(S, \phi)}, \xi_{(S, \phi)})$ by $(M, \xi)$. 
We naturally extends the projection $\pi:\tilde S\to S$ to a map 
$P: \tilde S \times [0,1] \to S \times [0,1]$  
between the product manifolds such that the restriction of $P$ to each page $\tilde S_t (\simeq \tilde S)$ satisfies $P|_{\tilde S_t}=\pi$. 
By the commutativity $\pi \circ \tilde\phi = \phi \circ \pi$ 
we have 
$\phi(S_1) = \phi \circ \pi(\tilde S_1) = \pi \circ \tilde\phi(\tilde S_1) = \pi (\tilde S_0) = S_0$ 
thus the map $P:\tilde S \times [0,1] \to S \times [0,1]$ 
extends to the mapping tori 
$P: (\tilde S \times [0,1])/\tilde\phi \to (S \times [0,1])/\phi$ 
and then over to the bindings. 
Namely, the map $P$ 
induces a covering map $P: \tilde M  \longrightarrow  M$.

Let $\alpha$ be a contact 1-form on $M$ such that $\xi = \ker\alpha$. 
Let $\tilde\alpha := P^*\alpha$ be the pullback of $\alpha$  then $\tilde\alpha\wedge d\tilde\alpha = P^*(\alpha\wedge d\alpha)>0$ and $\ker\tilde\alpha$ gives a contact structure on $\tilde M$ such that $P_*(\ker\tilde\alpha) = \ker\alpha=\xi$. 
This shows that $P:(\tilde M, \ker\tilde\alpha)\to (M,\xi)$ is a covering map. 
We also see that $(\tilde M, \ker\tilde\alpha)$ is supported by the open book $(\tilde S, \tilde\phi)$, that is, 
$\tilde\alpha>0$ on the binding of the open book $(\tilde S, \tilde\phi)$ and $d\tilde\alpha>0$ on each page $\tilde S_t$.  Thus the Giroux correspondence implies that $(\tilde M, \ker\tilde\alpha)$ and $(\tilde M, \tilde \xi)$ are isotopic, and the map $(\tilde M, \tilde\xi)\stackrel{P}{\to} (M,\xi)$ is a covering map. 
\end{proof}

Conversely, we have the following. 

\begin{theorem}\label{thm2}
Let $P:(\tilde M, \tilde\xi)\to(M,\xi)$ be a covering map for contact manifolds. 
For every open book $(S,\phi)$ supporting $(M,\xi)$ there exists an open book $(\tilde S, \tilde\phi)$ supporting $(\tilde M, \tilde\xi)$ and giving an open book covering map $\pi:(\tilde S, \tilde\phi)\to (S, \phi)$ compatible with $P$. 
\end{theorem}

\begin{proof}
Let $S_t$ ($t\in[0,1]$) denote the pages of the open book decomposition $(S, \phi)$ of $M$. 
Let $\tilde S_t := P^{-1}(S_t)$ and $\tilde B = P^{-1}(B)$, where $B \subset M$ is the binding for $(S,\phi)$. 
All the $\tilde S_t$ have the same topological type, denoted by $\tilde S$, and $P$ induces a covering map $\pi:\tilde S\to S$. 
There exists $\tilde\phi \in \Aut(\tilde S, \partial\tilde S)$ such that: 
$$
\tilde M\setminus\tilde B \simeq 
(\tilde S\times[0,1]) / (x,1) \sim (\tilde\phi(x), 0)
$$
Since the pages $S_0$ and $S_1$ are identified under $\phi$ the commutativity $\pi \circ \tilde\phi=\phi\circ\pi$ holds. 
Thus we get an open book covering map $\pi:(\tilde S, \tilde\phi)\to (S, \phi)$ compatible with $P$. 

By the same argument as in the proof of Proposition~\ref{prop1} we can show that $(\tilde S, \tilde\phi)$ supports the contact manifold $(\tilde M, \tilde\xi)$. 
\end{proof}

\begin{remark}
For a covering map of contact 3-manifolds $P:(\tilde M, \tilde \xi) \rightarrow (M,\xi)$, not every open book decomposition $(\tilde S, \tilde \phi)$ of $(\tilde M, \tilde \xi)$ arises as an open book covering compatible with $P$. 
\end{remark}
To see this statement we recall the following simple fact, which easily follows from the definition of right-veering diffeomorphisms \cite{HKM}.
\begin{lemma}
\label{lemma:coverrv}
Let $\pi:(\tilde S, \tilde \phi) \rightarrow (S,\phi)$ be an open book covering map. Then $\phi$ is right-veering if and only if $\tilde \phi$ is right-veering.
\end{lemma}

Now consider the case that $(M,\xi)$ is tight and $(\tilde M, \tilde \xi) $ is overtwisted. 
Then by \cite[Theorem 1.1]{HKM} there is an open book decomposition $(\tilde S, \tilde \phi)$ of $(\tilde M, \tilde \xi)$ such that $\tilde \phi$ is not right-veering. 
On the other hand, since $(M,\xi)$ is tight every open book $(S,\phi)$ of $(M,\xi)$ has right-veering $\phi$.
Hence Lemma \ref{lemma:coverrv} shows that the non-rightveering open book $(\tilde S, \tilde \phi)$ cannot cover $(S,\phi)$.

\section{The overtwisted complexity, depth of bindings and open book coverings}\label{sec3}

In this section we study properties of open book coverings using the notion of {\em right-veeringness} \cite{HKM} and the open book foliation method \cite{ik1-1}.

Let us recall the {\em overtwisted complexity} $n(S, \phi)$ introduced in \cite[Definition 6.4]{ik2}.
It is a non-negative integer given by:
$$n(S, \phi) = \min\left\{
e_-(\F(D)) \ | \ 
D \mbox{ is a transverse overtwisted disk in } (S, \phi)\right\}, 
$$
if $(S, \phi)$ supports an overtwisted contact structure, and  $n(S, \phi)=0$ otherwise.
Here $e_-(\F(D))$ denotes the number of negative elliptic points in the open book foliation on $D$. 
See Definition 4.1 of \cite{ik1-1} for the definition of a {\em transverse overtwisted disk}, which can be understood as a transverse push-off of a usual overtwisted disk, or, the spanning disk of a transverse unknot $K$ with $sl(K)=+1$. 
 
The following property is proved in \cite{ik2}. 
\begin{proposition}\label{prop:n}
\cite[Corollary 6.5]{ik2} 
\begin{enumerate}
\item
$n(S, \phi) = 0$ if and only if $\xi_{(S, \phi)}$ is tight (and hence $\phi$ is right veering). 
\item 
$n(S, \phi) = 1$ if and only if $\xi_{(S, \phi)}$ is overtwisted and $\phi$ is not right veering. 
\item 
$n(S, \phi) \geq 2$ if and only if $\xi_{(S, \phi)}$ is overtwisted and $\phi$ is right veering. 
\end{enumerate}
\end{proposition}

As a consequence we can show the following:

\begin{proposition}\label{Cor65}
Let $\pi: (\tilde S, \tilde\phi)\to(S,\phi)$ be an open book covering such that $n(\tilde S, \tilde\phi)=1$ then $(S, \phi)$ supports an overtwisted contact structure.
\end{proposition}

\begin{proof}
Suppose that $(S,\phi)$ supports a tight contact structure. 
Then $\phi$ is right-veering for every boundary component of $S$. 
By Lemma~\ref{lemma:coverrv} $\tilde\phi$ is also right-veering for every boundary component of $\tilde S$. 
The property (3) of Proposition~\ref{Cor65} implies that $n(\tilde S, \tilde\phi)\geq 2$, which is a contradiction. 
\end{proof}

The overtwisted complexity is closely related to the depth of transverse knots or links introduced by Baker and Onaran in \cite{bo}: 
The \emph{depth} of a transverse knot or link\footnote{As mentioned in Remark 5.2.4 of \cite{bo} the depth can be defined for links though it is originally defined for knots.} 
$K$ in an overtwisted contact 3-manifold $(M,\xi)$ is defined by 
\[ 
d(K)= \min \{ |D \cap K| \: | \: D \text{ is an overtwisted disk in } (M,\xi)   \} 
\]
and $K$ is called \emph{non-loose} if $d(K)>0$, that is, $\xi$ is tight on ${M\setminus K}$. 

\begin{theorem}\label{thm:d=n}
Let $B$ be the binding of an open book $(S,\phi)$ supporting an overtwisted contact structure.
If $(S, \phi)$ supports an overtwisted contact structure then $d(B)=n(S,\phi)$. 
\end{theorem}

\begin{proof}
Let $D_{\sf trans}$ be a transverse overtwisted disk realizing $n(S,\phi)$, that is, the open book foliation $\F(D_{\sf trans})$ has $n(S,\phi)$ negative elliptic points.

Let $(B, \pi)$ be an open book decomposition of $M_{(S, \phi)}$ that is determined by the abstract open book $(S, \phi)$.  
By \cite[Theorem 2.21]{ik1-1} we may choose a contact structure $\xi$ supported by $(B, \pi)$ such that the characteristic foliation $\cF(D_{\sf trans})$ and the open book foliation $\F(D_{\sf trans})$ are topologically conjugate. 
Moreover we may assume that 
the set of positive/negative elliptic points of $\cF(D_{\sf trans})$ coincides exactly with the set of positive/negative elliptic points of $\F(D_{\sf trans})$.

Recall that a positive/negative elliptic point of the open book foliation on a surface $F$ is just a positive/negative intersection point of $F$ and the binding $B$. 
With this in mind we denote by $B \cap^{\pm}  D_{\sf trans}$ the set of $\pm$-intersection points of $D_{\sf trans}$ and $B$.

Let $B'$ be a transverse link that is obtained from $B$ by transverse isotopy {\em only near} the 
intersection points $B \cap D_{\sf trans}$ so that
\begin{itemize}

\item
$|B' \cap^+ D_{\sf trans}| = |B \cap^+ D_{\sf trans}|$ 
and 
$(B' \cap^+ D_{\sf trans}) \subset A$
\item 
$|B' \cap^- D_{\sf trans}| = |B \cap^- D_{\sf trans}|$
and $B' \cap  G_{--}(\cF(D_{\sf trans})) = \emptyset$
\end{itemize}
where $A \subset D_{\sf trans}$ is the annulus bounded by the graph $G_{++}(\cF(D_{\sf trans}))$ and the boundary $\partial D_{\sf trans}$, and 
$G_{++}(\cF(D_{\sf trans}))$ (resp. $G_{--}(\cF(D_{\sf trans}))$) is the Giroux graph in the characteristic foliation consisting of positive (resp. netagive) elliptic points and stable (resp. unstable) separatrices of positive (resp. negative) hyperbolic points (see \cite[page~646]{G} and \cite[Definition 2.17]{ik1-1}). 
Since the two foliations $\cF(D_{\sf trans})$ and $\F(D_{\sf trans})$ are topologically conjugate the graphs $G_{\pm\pm}(\cF(D_{\sf trans}))$ and $G_{\pm\pm}(\F(D_{\sf trans}))$ are topologically conjugate.  
By the definition of a transverse overtwisted disk \cite[Definition 4.1]{ik1-1} the graph $G_{--}$ is a tree and $G_{++}$ is a circle enclosing $G_{--}$. See Figure \ref{fig:transtousual}, where $G_{--}(\cF(D_{\sf trans}))$ and $G_{++}(\cF(D_{\sf trans}))$ are depicted by the gray and the black bold arcs, respectively.  
\begin{figure}[htbp]
\begin{center}
\SetLabels
(0*.05) $\cF(D_{\sf trans})$\\
(.75*.66) $D$\\
(.17*.85) $+$\\
(.1*.23) $G_{++}$\\
(.2*.55) $-$\\
(.2*.45) $G_{--}$\\
(1*.05) $\cF(D_{\sf trans}')$\\
(.33*.85) $A$\\
\endSetLabels
\strut\AffixLabels{\includegraphics*[height=52mm]{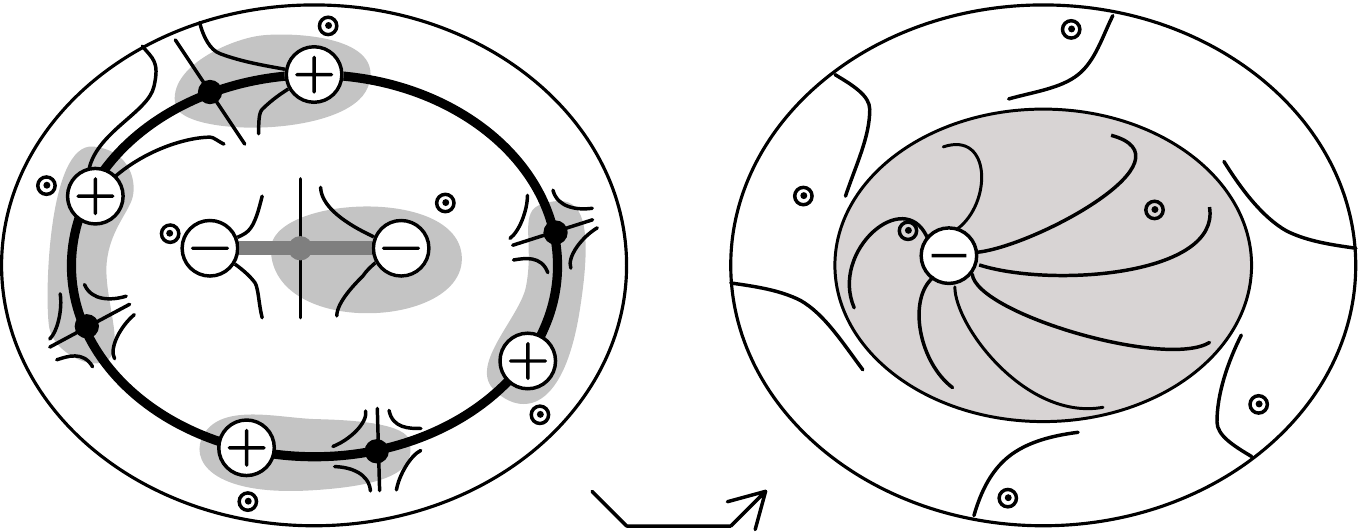}}
\caption{Giroux elimination lemma is applied to the gray regions in the left disk. 
The dots $\odot$ represent the intersection points  $B' \cap D_{\sf trans}$ and $B' \cap D'_{\sf trans}$.}
\label{fig:transtousual}
\end{center}
\end{figure}
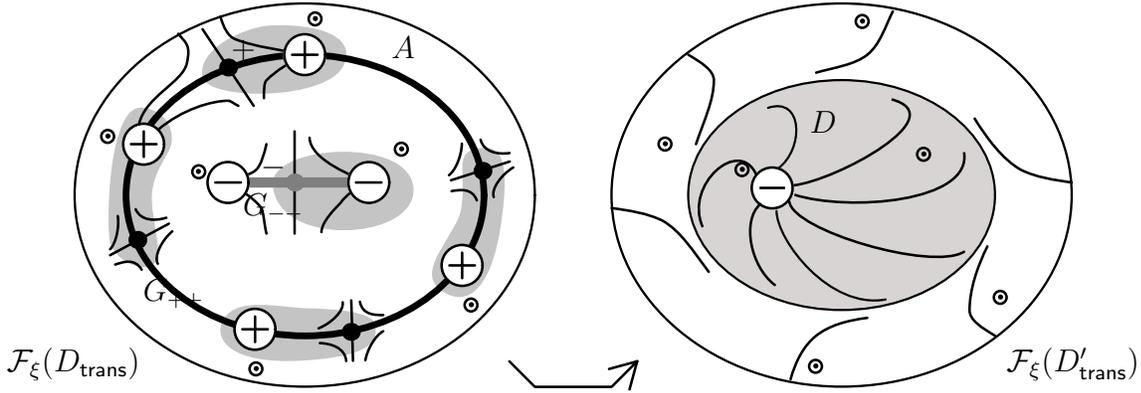

Note that $B'$ is not used as a binding but it is just a transverse link. 
We also keep using the same contact structure $\xi$, hence the characteristic foliation $\cF(D_{\sf trans})$ does not change.

We apply the {\em Giroux elimination lemma} \cite[Lemma 3.3]{Gconvex} to small 3-ball neighborhoods (gray regions in Figure~\ref{fig:transtousual}) of $G_{\pm\pm} (\cF(D_{\sf trans}))$ each of which contains a pair of consecutive elliptic and hyperbolic points (of the same sign) and is disjoint from $B'$. 
We can find a disk, $D_{\sf trans}'$, and a sub-disk, $D$, of $D_{\sf trans}'$ with the following properties: 
\begin{itemize}
\item
$D_{\sf trans}'$ is $C^0$ close to $D_{\sf trans}$.
\item 
$D$ is a standard overtwisted disk, i.e., its characteristic foliation contains exactly one elliptic singularity and ${\rm tb}(\partial D)=0$. 
\item
$\{B' \cap^+ D_{\sf trans}'\} = \{B' \cap^+ D_{\sf trans}\}$ 
and 
$|B' \cap^+ D| = 0$.
\item
$\{B' \cap^- D_{\sf trans}\} = \{ B' \cap^- D_{\sf trans}'\} =\{B' \cap^- D\}$ 
\end{itemize}
Here the third property follows from the condition $(B' \cap^{+} D_{\sf trans}) \subset A$. 
The fourth property follows from the condition $B' \cap  G_{--}(\cF(D_{\sf trans})) = \emptyset$. 
Though $D_{\sf trans}'$ may not admit an open book foliation this would not be a problem.
We have 
$$d(B)=d(B') \leq |B' \cap D| = |B' \cap^- D| = 
|B' \cap^- D_{\sf trans}| =
|B \cap^- D_{\sf trans}| = n(S, \phi).$$
Thus $d(B) \leq n(S,\phi)$.

Conversely, let $D$ be an overtwisted disk realizing $d(B)$, that is, $|B \cap D| = d(B)$. 
Taking the positive transverse push-off of the Legendrian boundary $\partial D$ we find a transverse unknot, $K$, with $sl(K)=1$. 
A spanning disk $D'$ of $K$ still intersects $B$ at $d(B)$ points. 
By Pavalescu's proof of Alexander theorem \cite[Theorem 3.2]{pav}, there is an isotopy preserving each page
and moving the non-braided parts of $K$ to neighborhoods of the binding. 
In the neighborhoods we can move $D'$ so that $K=\partial D'$ becomes a closed braid without introducing negative intersection points of $D'$ and $B$. 

Following the discussion in the proof of \cite[Theorem 4.3]{ik1-1}, from $D'$ we can construct a transverse overtwisted disk $D_{\sf trans}$ whose open book foliation has no more than $d(B)$ negative elliptic points, hence $n(S,\phi) \leq d(B)$.
\end{proof}

As a consequence of Proposition~\ref{prop:n} and Theorem~\ref{thm:d=n} we have the following  characterization of depth one bindings, which generalizes \cite[Theorem 5.2.3]{bo} (except for the part regarding the tension invariant). 

\begin{corollary}\label{corA}
Let $B$ be the binding of an open book $(S,\phi)$. Then $d(B)=1$ if and only if $\phi$ is not right-veering. 
\end{corollary}

Corollary \ref{corA} gives a construction of Legendrian or transverse knots and links with large depth (cf. \cite[Problems 6.1 and 6.4]{bo}).

\begin{corollary}\label{corB}
Let $B$ be the binding of an open book $(S,\phi)$ supporting an overtwisted contact structure.
Let $L$ be a Legendrian approximation of $B$. 
If $\phi$ is right-veering then $1 < d(B) \leq d(L)$. 
\end{corollary}

The inequality $d(B) \leq d(L)$ holds even without the right-veering assumption.

In fact, there are several constructions of 
right-veering open books supporting overtwisted contact structures as listed below:
\begin{enumerate}
\item \cite[Proposition 6.1]{HKM}
Every open book can be made right-veering after a sequence of positive stabilizations. 
\item 
By Theorem \ref{thm2}, for a covering map $P:(\tilde M, \tilde \xi) \rightarrow (M,\xi)$ between a tight $(M,\xi)$ and an overtwisted $(\tilde M, \tilde \xi)$ with an open book $(S,\phi)$ supporting $(M,\xi)$, there exists an open book covering $\pi:(\tilde S, \tilde \phi) \rightarrow (S,\phi)$ compatible with $P$. 
By \cite[Theorem 1.1]{HKM} $\phi$ is right-veering and Lemma~\ref{lemma:coverrv} implies that $\tilde \phi$ is right-veering. 

Such a family of examples is discussed in Proposition~\ref{key-example} where $d(\tilde B)=2$.
\end{enumerate}

If the bindings of a open book is not connected then by further positive stabilizations, which preserve the right-veering property, we can always make the binding connected.
Hence it is fairly easy to construct a transverse or Legendrian knot with depth greater than $1$.

We point out that if an open book $(\tilde S, \tilde \phi)$ in 
the construction (2) is not destabilizable then it gives rise to an example of right-veering, non-destabilizable open book supporting an overtwisted contact structure.
The existence (or non-existence) of such open books is asked in \cite{HKM} and many examples have been found \cite{le, lis, ik1-2, kr}.

Presumably, under certain condition, open book coverings would provide non-destabilizable open books: 
In \cite{etl} it is shown that a right-veering open book $(S,\phi)$ is destabilizable if and only if the {\em translation distance} (see \cite{etl} for the definition) of $\phi$ is equal to one. 
Although the behavior of the translation distance under a covering operation is not clear, it is likely that if $\phi$ has a large translation distance then so does $\tilde\phi$, hence open book covering is non-destabilizable.




\section{Illustration of overtwisted coverings and a pants pattern}\label{sec4}

In this section we study a sequence of open books that supports overtwisted, virtually overtwisted tight and universally tight contact structures.

We begin with a proof of Proposition~\ref{key-example}. 

\begin{proof}
We prove the assertion (1). 
Applying the proof of Theorem 4.1 in \cite{ik1-2} we can construct a transverse overtwisted disk in the open book $(S, \phi)$. 
By the definition every transverse overtwisted disk has the self-linking number $1$, that is, the Bennequin-Eliashberg inequality \cite{E2} is violated. 
Thus $(S, \phi)$ supports an overtwisted contact structure. 

The assertion (3) follows from the same argument in Example 5.2 of \cite{EV}.  

Finally we prove the assertion (2).
By the lantern relation (see for example \cite[Proposition 5.1]{FM}) the mapping class $\phi$ can be written in the product of positive Dehn twists. 
Therefore, results of Giroux \cite{G} and  Eliashberg-Gromov \cite{EG} imply that $(M, \xi)$ is tight.

Below we consider the following four cases. We find a transverse overtwisted disk in an open book covering for each case.\\
(Case 1) $p-1\equiv q-1\equiv 1$ (mod $2$); \\
(Case 2) $p-1 \equiv q-1\equiv 0$ (mod $2$); \\
(Case 3) $p-1 \equiv 1$ and $q-1 \equiv 0$ (mod $2$);\\
(Case 4) $p-1 \equiv 0$ and $q-1 \equiv 1$ (mod $2$):

For each case we cut two copies of $S$ along the thick gray arcs as shown in Figure~\ref{cut}, then glue them along the cut arcs to get a connected surface $\tilde S$. 
Clearly $\tilde S$ is a double cover of $S$. 
We call the projection map $\Pi:\tilde S\to S$.  
(For Cases 3 and 4, the base space $S$ is disconnected after the cut but it is easy to verify that the covering space $\tilde S$ is connected.) 
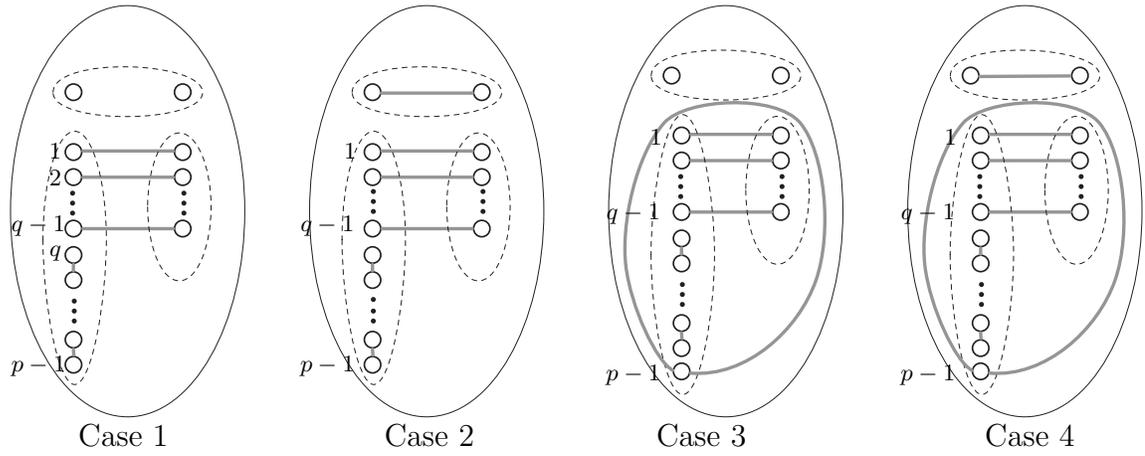
\begin{figure}[htbp]
\begin{center}
\SetLabels
(.1*-.07) Case 1\\
(.04*.625) \scriptsize $1$\\
(.04*.56) \scriptsize $2$\\
(.025*.45) \scriptsize $q-1$\\
(.04*.39) \scriptsize $q$\\
(.37*-.07) Case 2\\
(.025*.11) \scriptsize $p-1$\\
(.3*.625) \scriptsize $1$\\
(.28*.45) \scriptsize $q-1$\\
(.28*.11) \scriptsize $p-1$\\
(.61*-.07) Case 3\\
(.57*.66) \scriptsize $1$\\
(.55*.48) \scriptsize $q-1$\\
(.55*.09) \scriptsize $p-1$\\
(.9*-.07) Case 4\\
(.83*.66) \scriptsize $1$\\
(.81*.48) \scriptsize $q-1$\\
(.81*.09) \scriptsize $p-1$\\
\endSetLabels
\strut\AffixLabels{\includegraphics*[height=55mm]{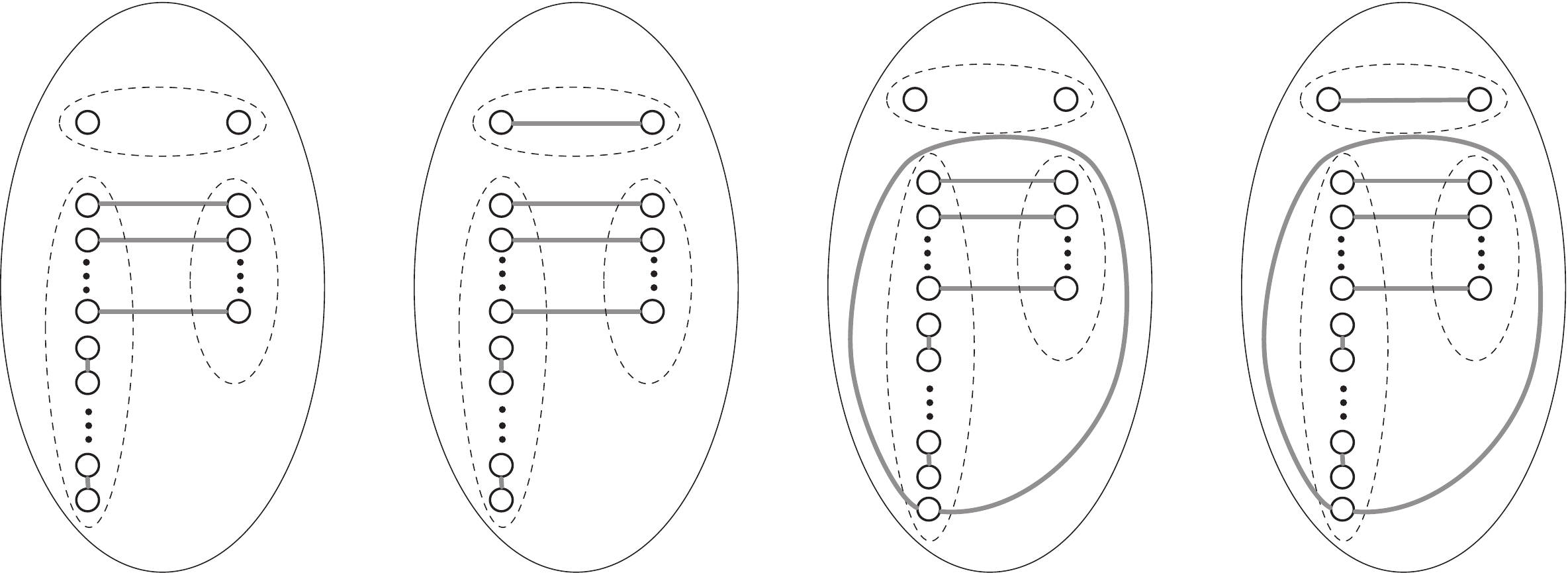}}
\caption{Cutting arcs (highlighted gray) in $S$ to construct $\tilde S$.}\label{cut}
\end{center}
\end{figure}
%

Choose base points $x_0 \in S$ and $\tilde x_0 \in \Pi^{-1}(x_0)$. 
Let $G$ be the index two subgroup of $\pi_{1}(S, x_0)$  defined by $G = \{\gamma \in \pi_{1}(S, x_0)\: | \: \langle [\gamma], [c]\rangle = 0\}$, where
$\langle -,- \rangle: H_{1}(S)\times H_{1}(S,\partial S) \rightarrow \Z_2$ is the mod $2$ algebraic intersection pairing and $[c] \in H_{1}(S,\partial S)$ is the relative homology class represented by the set of cutting arcs (with any choice of orientation). 
Note that the covering space $\Pi: \tilde{S} \to S$ has $\Pi_*(\pi_1(\tilde S, \tilde x_0))=G$. 
Since $\phi_*(G)=G$ there is a homeomorphism $\tilde\phi: \tilde S\to \tilde S$ such that $\tilde\phi(\tilde x_0)=\tilde x_0$ and $\Pi \circ \tilde\phi = \phi\circ\Pi$. 
We call $\tilde\phi$ a lift of $\phi$. 
For an arc $\tilde{\gamma}$ in $\tilde{S}$ the image $\tilde \phi(\tilde{\gamma})$ is nothing but the lift of the arc $\phi(\pi(\tilde{\gamma}))$ in $S$. This allows us to compute $\tilde \phi$ and one can check that $\tilde \phi$ fixes the boundary $\partial\tilde S$ pointwise.
In general it may be hard to write $\tilde\phi$ as the product of Dehn twists.

Figure~\ref{OTcover-odd2} (resp. Figure~\ref{OTcover-even2}) gives a movie presentation of a transverse overtwisted disk for Case 1 (resp. Case 2). 
For Case 3 and Case 4 combining the ideas of Figures~\ref{OTcover-odd2} and \ref{OTcover-even2} one can also find transverse overtwisted disks. 
We leave it to readers as an exercise. 
Therefore, the open book $(\tilde S, \tilde\phi)$ supports an overtwisted contact structure. 
\end{proof}

\begin{figure}[htbp]
\begin{center}
\SetLabels
(.5*.97) ($t=0$)\\
(.5*.93) Case 1\\
(.75*.925) \scriptsize $(-)$\\
(.5*.62) ($t=t_1$)\\
(.5*.29) ($t=t_2$)\\
\endSetLabels
\strut\AffixLabels{\includegraphics*[height=21cm]{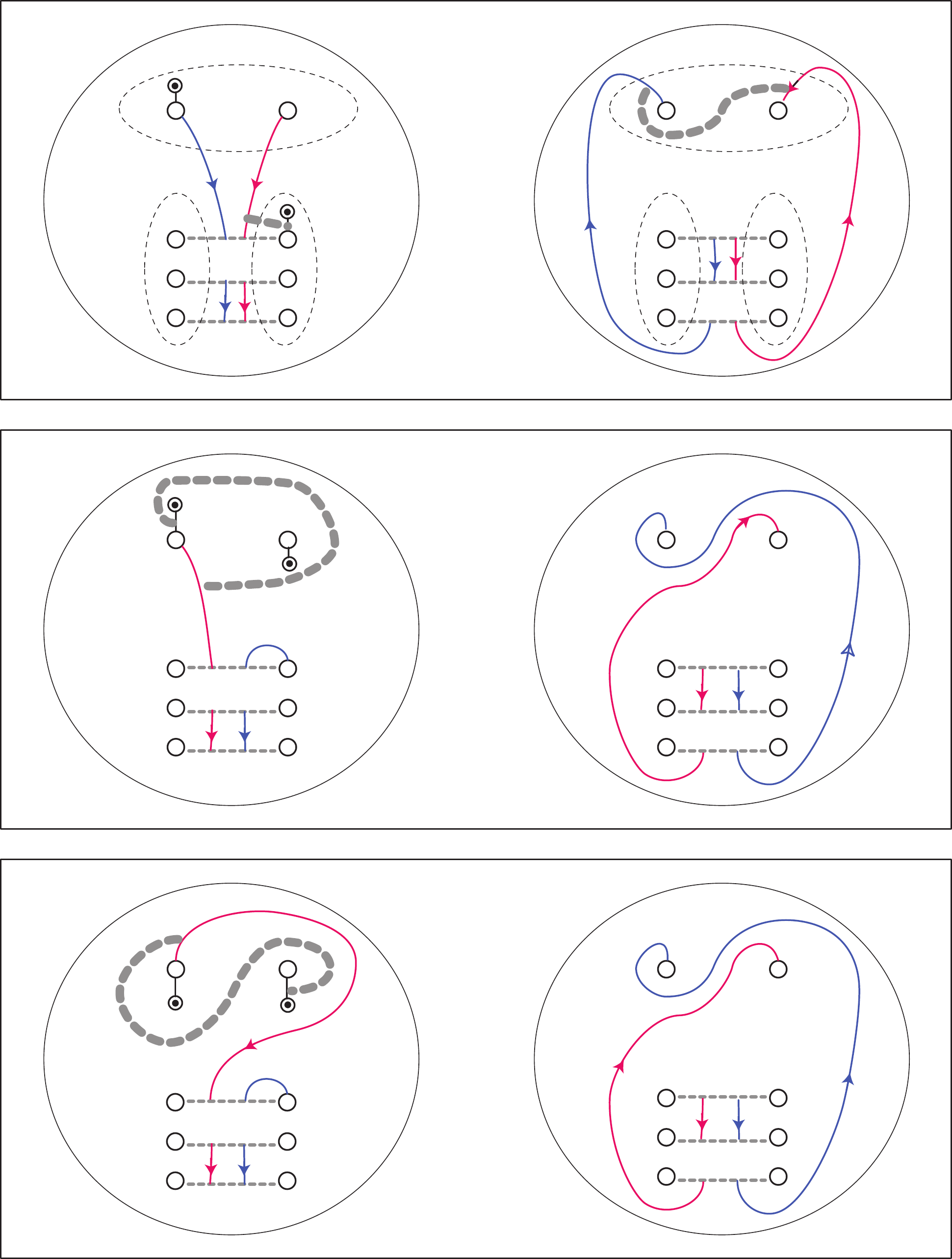}}
\end{center}
\end{figure}
\begin{figure}[htbp]
\begin{center}
\SetLabels
(.5*.95) ($t=t_3$)\\
(.5*.44) ($t=1$)\\
\endSetLabels
\strut\AffixLabels{\includegraphics*[height=13cm]{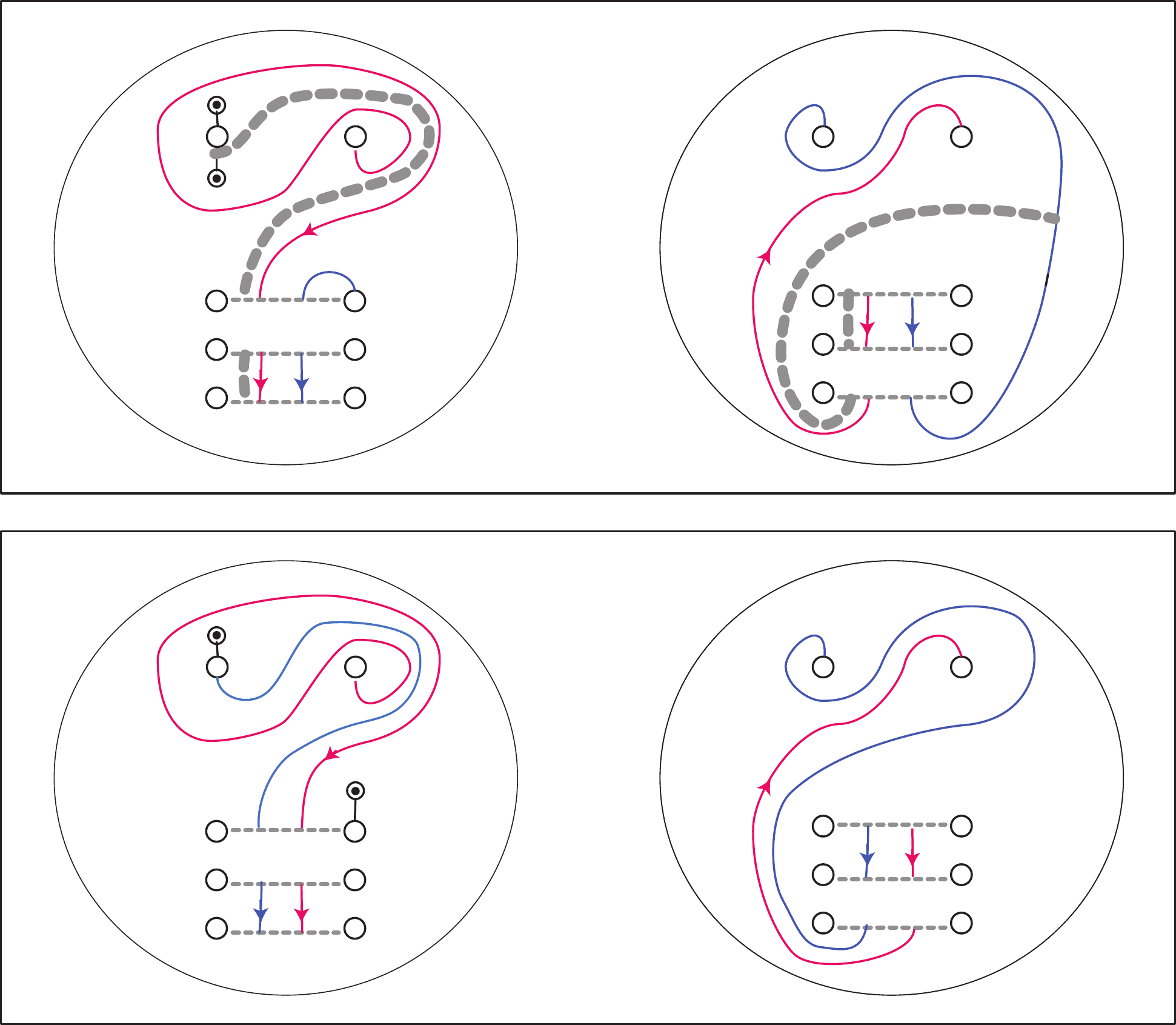}}
\caption{(Case 1, $p-1=q-1=3$): A movie presentation of a  transverse overtwisted disk. 
Each arrow indicates orientation of the b-arc. 
The starting (resp. ending) point of a b-arc is a positive (resp. negative) elliptic point. 
The end point of an a-arcs is marked with a $\odot$. 
This transverse overtwisted disk has two negative elliptic points and four positive elliptic points. 
Thick dashed arcs are describing arcs for hyperbolic points. 
The signs of the hyperbolic points are all positive except the one marked with $(-)$ in the page $t=0$. 
One can easily generalize this to any $p$ and $q$ with $p-1\equiv q-1\equiv 1$ (mod $2$).}
\label{OTcover-odd2}
\end{center}
\end{figure}
\begin{figure}[htbp]
\begin{center}
\SetLabels
(.5*.97) ($t=0$)\\
(.5*.93) Case 2\\
(.5*.62) ($t=t_1$)\\
(.5*.29) ($t=t_2$)\\
\endSetLabels
\strut\AffixLabels{\includegraphics*[height=21cm]{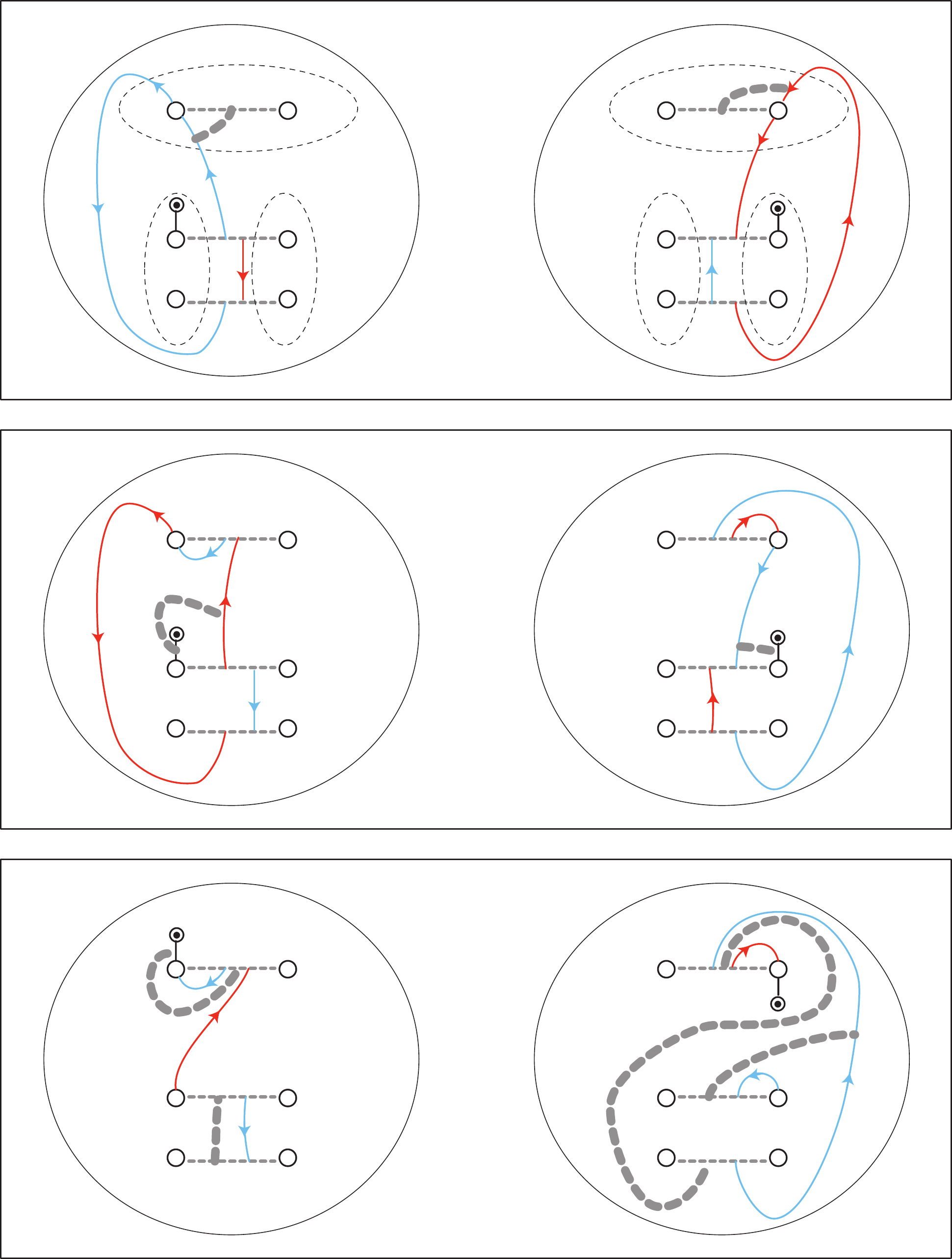}}
\end{center}
\end{figure}
\begin{figure}[htbp]
\begin{center}
\SetLabels
(.5*.95) ($t=t_3$)\\
(.5*.44) ($t=1$)\\
\endSetLabels
\strut\AffixLabels{\includegraphics*[height=13cm]{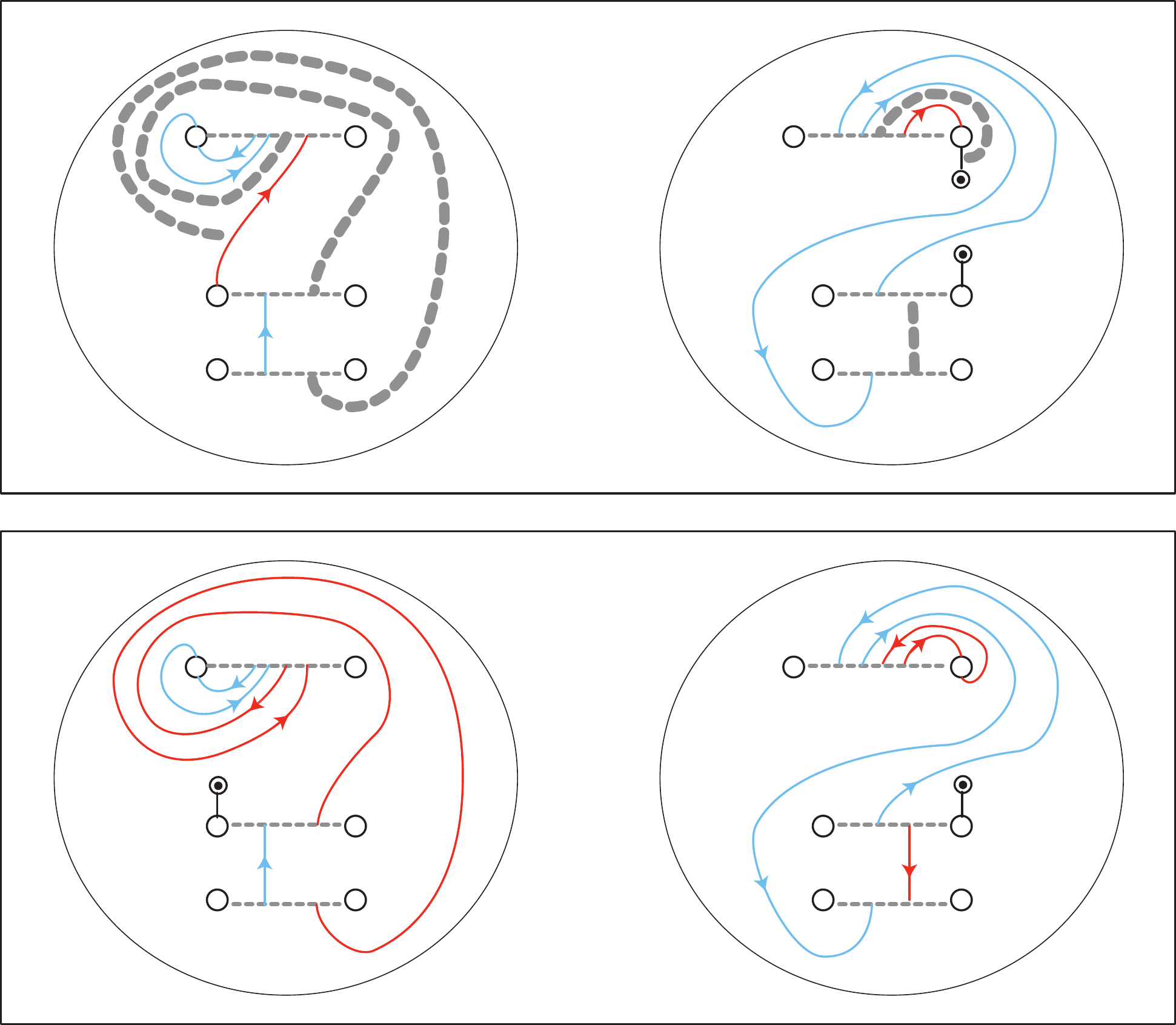}}
\caption{(Case 2, $p-1=q-1=2$): A movie presentation of a  transverse overtwisted disk. 
This transverse overtwisted disk has two negative elliptic points and four positive elliptic points.  
The signs of the hyperbolic points are all positive except the one in the page $t=0$. 
One can 
generalize this to any $p$ and $q$ with $p-1\equiv q-1\equiv 0$ (mod $2$)}
\label{OTcover-even2}
\end{center}
\end{figure}

\pagebreak

\begin{remark}
Let $K_{p,q} \subset (S^3, \xi_{st})$  be a Legendrian unknot in the standard contact $3$-sphere with the Thurston-Bennequin number ${\rm tb}(K_{p,q})=-(p+q)+1$ and the rotation number ${\rm rot}(K_{p,q})=p-q$  or $q-p$, where $p, q \geq 2$. 
See Figure~\ref{FrontProj}. 
\begin{figure}[htbp]
\begin{center}
\SetLabels
(0*.8) $p$\\
(1*.8) $q$\\
(0*.17) $1$\\
(0*.3) $2$\\
(1*.17) $1$\\
(1*.3) $2$\\
\endSetLabels
\strut\AffixLabels{\includegraphics*[height=30mm]{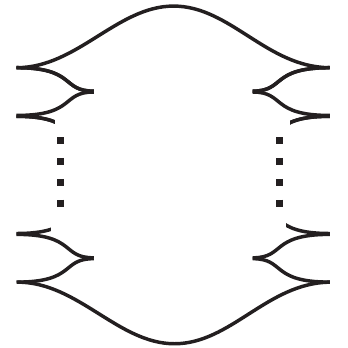}}
\caption{The front projection of $K_{p,q}$.}\label{FrontProj}
\end{center}
\end{figure}
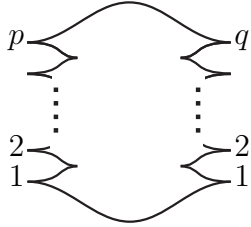

Let $(M, \xi)$ denote the contact structure obtained from $(S^3, \xi_{st})$ by the Legendrian surgery along $K_{p,q}$. 
When $n=-1$ in Proposition~\ref{key-example} we can verify that the open book $(S, \phi)$ supports the contact manifold $(M, \xi)$ applying Sch\"onenberger's algorithm \cite{S} and the lantern relation. 

With this identification of $(S, \phi)$ and $(M, \xi)$ the assertion (2) of Proposition~\ref{key-example}  
can also be proved applying Gompf's criterion of virtually overtwisted contact structures \cite{Go}. 

Lastly, we note that if $p=1$ or $q=1$ then $\xi$ is known to be universally tight due to Honda \cite{H} and Giroux \cite{G1}. 
\end{remark}

\begin{remark}
Under the projection $\pi:\tilde S_t \to S_t$ for each page we can see how the transverse overtwisted disk is `folded' (in other words, self-intersecting) in the base tight manifold $(M, \xi)$. 
For example when $t=t_1$ of Case 1 in Figure~\ref{OTcover-odd2} the projected image of the transverse overtwisted disk has two intersection points marked with black dots as in Figure~\ref{projection}.
\begin{figure}[htbp]
\begin{center}
\SetLabels
(1*0) page $S_{t_1}$\\
\endSetLabels
\strut\AffixLabels{\includegraphics*[height=35mm]{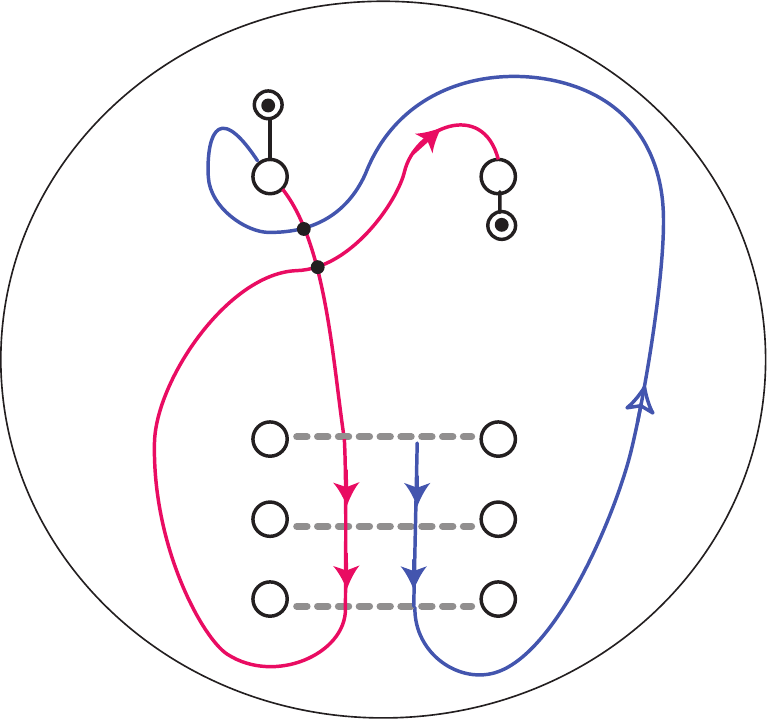}}
\caption{Self-intersection points of the transverse overtwisted disk (of Figure~\ref{OTcover-odd2}) under the projection $\pi:\tilde S_t \to S_t$. }\label{projection}
\end{center}
\end{figure} 
\end{remark}

One can generalize the construction of overtwisted disks to a $k$-fold cover using the same cutting arcs of Figure~\ref{cut}. 
For example: 

\begin{example}\label{ex1}
Let $S$ be a $2$-sphere with four holes. Let $a,b,c,d,e \subset S$ be simple closed curves parallel to the boundary as shown in Figure~\ref{4psphere}.
Let 
$$\Phi_{\alpha, \beta} = T_{a}^{\alpha+1}T_{b}^{\beta+1}T_{c}T_{d}T_{e}^{-1}$$
Suppose that $\alpha, \beta > 0$ and there exists a number $k\geq 2$ that divides both $\alpha+1$ and $\beta+1$.
Then there exists a $k$-fold cover of $(S, \Phi_{\alpha, \beta})$ that supports an overtwisted contact structure, i.e., $(S, \Phi_{\alpha, \beta})$ supports a  virtually overtwisted tight contact structure. 
\begin{figure}[htbp]
\begin{center}
\SetLabels
\endSetLabels
\strut\AffixLabels{\includegraphics*[height=35mm]{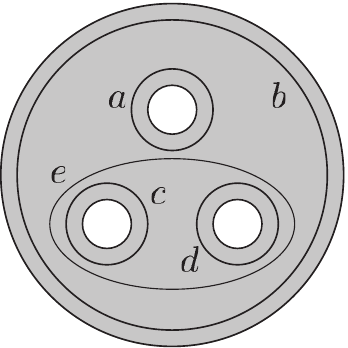}}
\caption{The surface $S$.}\label{4psphere}
\end{center}
\end{figure}
\end{example}

So far we have only seen planar open books. In fact  our example can be applied to higher genus open books: 
Suppose that an open book $(S', \Psi)$ supports a tight contact structure, the Nielsen-Thurston type of $\phi$ is reducible, and  `containing' $(S, \Phi_{\alpha, \beta})$ as a subspace, that is, $S \subset S'$ and $\Psi|_S = \Phi_{\alpha, \beta}$. 
Then $(S', \Psi)$ supports a virtually overtwisted contact structure.

\begin{example} 
Let $S'$ be a genus 4 surface with two holes, $d$ and $c$, see Figure~\ref{subspace}. 
Let $\Psi= T_{a}^{2}T_{b}^{2}T_{c}T_{d}T_{e}^{-1} T_f$ be a diffeomorphism of $S'$.  
The open book $(S', \Psi)$ contains $(S, \Phi_{1,1})$ 
of Example~\ref{ex1} and $(S', \Psi)$ supports a tight contact structure. 
Take a double cover $\tilde S'$ of $S'$ tjat os a genus 7 surface with four boundary components, $\tilde c, \tilde d, \tilde c', \tilde d'$. 
The monodromy $\Psi$ lifts to a diffeomorphism $\tilde\Psi= T_{\tilde a}T_{\tilde b}T_{\tilde c}T_{\tilde d}T_{\tilde e}^{-1} T_{\tilde f}T_{\tilde c'}T_{\tilde d'}T_{\tilde e'}^{-1} T_{\tilde f'}$ of $\tilde S'$.  
Example~\ref{ex1} guarantees that the open book $(\tilde S', \tilde\Psi)$ supports an overtwisted contact structure. 
\end{example}
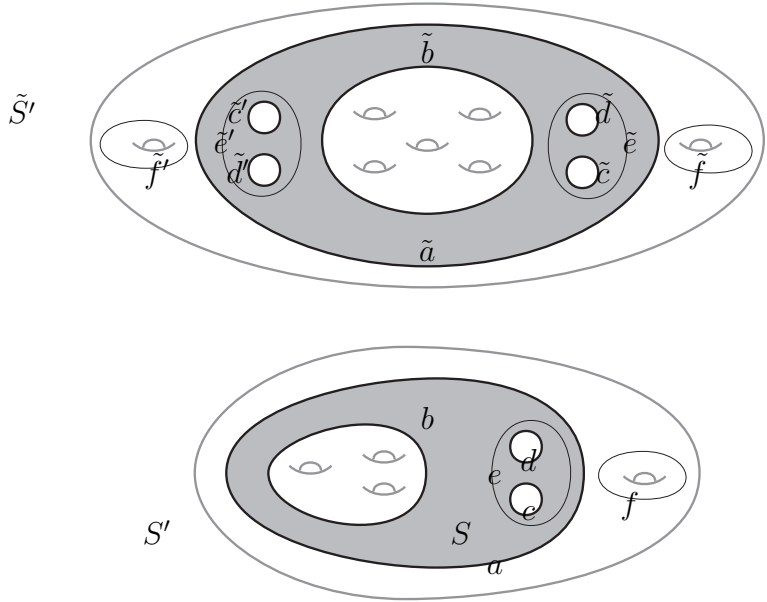
\begin{figure}[htbp]
\begin{center}
\SetLabels
(.55*.1) $S$\\
(.1*.1) $S'$\\
(-.1*.8) $\tilde{S'}$\\
(.5*.29) $b$\\
(.6*.05) $a$\\
(.6*.2) $e$\\
(.65*.14) $c$\\
(.65*.22) $d$\\
(.8*.15) $f$\\
(.76*.8) $\tilde d$\\
(.76*.7) $\tilde c$\\
(.8*.75) $\tilde e$\\
(.5*.9) $\tilde b$\\
(.5*.57) $\tilde a$\\
(.22*.8) $\tilde c'$\\
(.22*.7) $\tilde d'$\\
(.2*.75) $\tilde e'$\\
(.9*.7) $\tilde f$\\
(.1*.7) $\tilde f'$\\
\endSetLabels
\strut\AffixLabels{\includegraphics*[height=80mm]{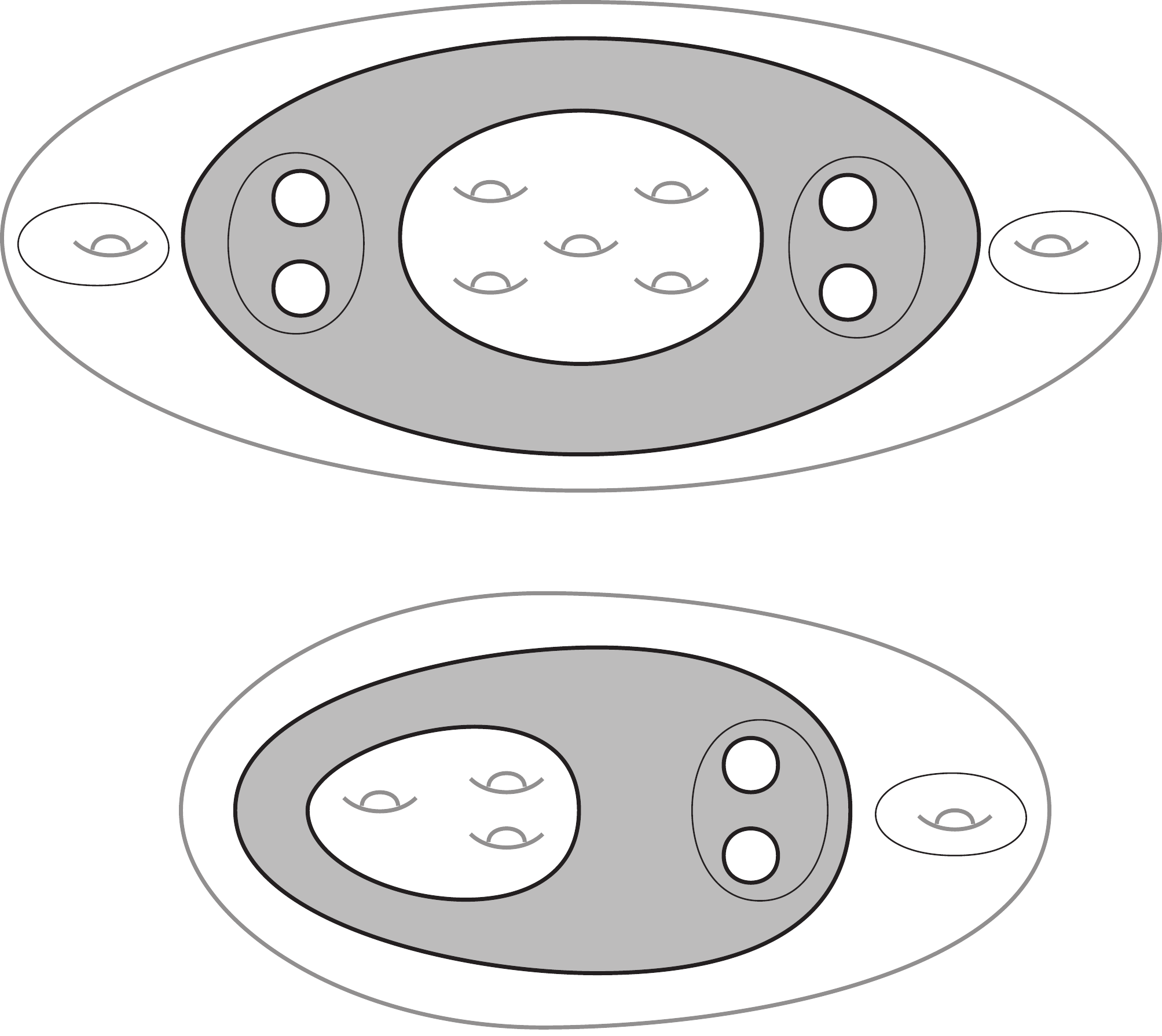}}
\caption{(Top) A genus 7 surface $\tilde S'$ with four holes.  (Bottom) An genus 4 surface $S'$ with two holes.}
\label{subspace}
\end{center}
\end{figure} 

\begin{remark}
Lemma \ref{lemma:coverrv} and the discussion in  Section~\ref{sec3} imply that if $(S, \phi)$ is a virtually overtwisted contact structure then its overtwisted cover has the overtwisted complexity (see Section~\ref{sec3}) $n(\tilde S, \tilde\phi) \geq 2$. 
We notice that all the examples of virtually overtwisted open books $(S, \phi)$ we study in this note have $n(\tilde S, \tilde\phi) = 2$.

Moreover, these open books $(S, \phi)$ all contain a pants region $P \subset S$ (see Figure~\ref{pants}) with the following properties 
\begin{itemize}
\item $P$ is bounded by curves $x, y, z$ with $x, y \subset \partial S$ and $z \subset {\rm Int}(S)$,
\item the monodromy $\phi$ preserves $P$ and $\phi|_P = T_x T_y {T_z}^{-1}$. 
\end{itemize}
 (The curve $z$ corresponds to $\alpha$ of Figure~\ref{sphere} and $e$ of Figure~\ref{4psphere}). 
Such a pants region $P$ plays a crucial role in our construction of transverse overtwisted disks  
because the two negative elliptic points of each transverse overtwisted disk lie on the lifts of $x$ and $y$. 
\begin{figure}[htbp]
\begin{center}
\SetLabels
(.23*.7) $x$\\
(.75*.7) $y$\\
(.5*.9) $z$\\
\endSetLabels
\strut\AffixLabels{\includegraphics*[height=20mm]{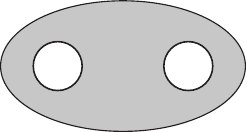}}
\caption{Pants region $P$.}
\label{pants}
\end{center}
\end{figure}
\end{remark}

\noindent{\bf Question:} 
Do there exist open book patterns, like the above pants pattern, that give virtually overtwisted contact structures? 


\section*{Acknowledgement}

The authors would like to thank John Etnyre for useful conversation on Theorem~\ref{thm2}, and John Etnyre, Jeremy Van Horn-Morris and Amey Kaloti for Proposition 4.1-(3). 
They also thank the referee for numerous comments that helped improving the paper significantly.   
TI was partially supported by JSPS KAKENHI
Grant Numbers 25887030 and 15K17540.
KK was partially supported by NSF grant DMS-1206770.

\end{document}